\documentclass[a4paper,leqno]{mathscan}

\usepackage[all]{xy}
\usepackage{graphicx}
 \usepackage{verbatim}
\usepackage[danish,english]{babel}
\usepackage{color}
\usepackage{enumerate}
\usepackage[abs]{overpic}

\newcommand{\R}{\mathbb R}
\newcommand{\N}{\mathbb N}
\newcommand{\Z}{\mathbb Z}

\newcommand{\im}{\mathrm{Im }}

\newcommand{\ve}{\underline e}
\newcommand{\vx}{\underline x}
\newcommand{\vy}{\underline y}
\newcommand{\vv}{\underline v}

\newtheorem{thm}{Theorem}[section]
\newtheorem{pro}[thm]{Proposition}

\newtheorem{lem}[thm]{Lemma}

\theoremstyle{definition}

\newtheorem{exam}[thm]{Example}     %%%%% by the journal

\usepackage[hidelinks]{hyperref}
\begin{document}

\title[Cohomology of the continuous hull]
 {Cohomology of the continuous hull of a combinatorial pentagonal tiling}

\author{Maria Ramirez-Solano}
\thanks{}

\address{Department of Mathematics, University of Copenhagen, Universitetsparken 5,
2100 K\o benhavn \O ,
 Denmark.}

\email{mrs@math.ku.dk}



\thanks{Supported by the Danish National Research Foundation through the Centre
for Symmetry and Deformation (DNRF92), and by the Faculty of Science of the University of Copenhagen.}

\begin{abstract}
In \cite{MRScontinuoushull} we constructed the continuous hull for the combinatorics of "A regular pentagonal
tiling of the plane", and in \cite{MRSinverselimit} we showed how we could write this hull as an inverse limit.
In this paper we show how to compute the cohomology of the hull and its topological K-theory.
\end{abstract}

\maketitle
The continuous hull $\Omega$ for a combinatorial pentagonal tiling defined in \cite{MRScontinuoushull} is written as an inverse limit of finite CW-complexes in \cite{MRSinverselimit}. Since cohomology reverses the arrows, and Cech cohomology behaves well under inverse limits, the Cech cohomology of the continuous hull is the direct limit of the Cech cohomology of the finite CW-complexes.
 This direct limit is computed by using the fact that the Cech cohomology of a finite CW-complex is the same as the cellular cohomology of the finite CW-complex, as the latter is easier to compute.
These steps are well-known for Euclidean tilings. See for instance \cite{PutnamBible95}, \cite{Sadun08}.

Cech Cohomology gives interesting information about the tiling space of an Euclidean tiling. For instance,  the zero cohomology group $H^0=\Z$ because it measures connected components and the continuous Hull has one connected component. Homology, on the other hand, gives useless information, for
$H_0$ is freely generated by an uncountable set and $H_n=0$, for  $n\in \N$. (See \cite{Sadun08} for more details).

The main result of this paper is Theorem \ref{t:cohomologyofOmega}, which tells us the Cech cohomology of the continuous hull $\Omega$.
We divide this paper into four sections.
In the first section, we give an overview of the cellular cohomology for a finite CW-complex.
The first problem that arises is the computation of the kernel of an integer matrix modulo the image of another one.
We take care of this problem in the second section using the Smith normal form of the matrix.
In the third section we give an overview of direct limits of sequences of $\Z$-modules and integer matrices, plus some techniques to compute some of them.
In the last section we compute the Cech cohomology of the continuous hull $\Omega$ using the formulas and definitions from the previous 3 sections.

\newpage

\section{General theory on cellular cohomology}
Suppose that $\Gamma$ is a finite CW-complex of dimension 2, with $F$ faces $\{f_1,\ldots,$ $f_F\}$, $E$ edges $\{e_1,\ldots,e_n\}$, and $V$ vertices $\{v_1,\ldots,v_V\}$.  The cellular chain complex
over $\Z$ for $\Gamma$ is defined as the sequence of maps $C_\bullet:=\{d_n:C_n \to C_{n-1} \}_{n\in\N_0}$, i.e.
$$\xymatrix{0\ar[r]&C_2\ar[r]^{d_2}&C_1\ar[r]^{d_1}&C_0\ar[r]&0},$$
 such that
\begin{itemize}
  \item $C_n=0$ for $n>2$.
  \item  $C_2$ is the $\Z$-module with basis $\{f_1,\ldots,f_F\}$. Hence $C_2\cong\Z^F$.
  \item  $C_1$ is the $\Z$-module with basis $\{e_1,\ldots,e_E\}$. Hence $C_1\cong\Z^E$.
  \item  $C_0$ is the $\Z$-module with basis $\{v_1,\ldots,v_V\}$. Hence $C_0\cong\Z^V$.
  \item The boundary map $d_2$ sends each face to the sum of its edges (there are $\pm$ signs taking care of orientation).
  \item The boundary map $d_1$ sends each edge to the difference of its vertices. Hence $d_1\circ d_2=0$.
\end{itemize}
An element of the kernel $Z_i(C_\bullet)\subset C_i$ is called a cycle and an element of the image $B_i(C_\bullet)\subset C_i$ is called a boundary.
We define the $i$-th homology group $H_i(C_\bullet):=Z_i(C_\bullet)/B_i(C_\bullet)$, which can be done because $C_i$ is in particular an Abelian group and the property $d_i\circ d_{i+1}=0$ imply that $B_i(C_\bullet)\subset Z_i(C_\bullet)$.
We define the sequence of  groups $H_\bullet(C_\bullet):=\{H_i(C_\bullet)\}_{i\in\N_0}$.
The cellular homology for the CW-complex $\Gamma$ is defined as $H_\bullet(\Gamma):=H_\bullet(C_\bullet)$.

The cellular cochain complex for $\Gamma$ is defined by
$$C^\bullet:=\mathrm{Hom}(C_\bullet,\Z):=\{d^n:\mathrm{Hom}_\Z(C_n,\Z)\to \mathrm{Hom}_\Z(C_{n+1},\Z)\}_{n\in\N_0},$$
where $d^n:=d_{n+1}^*$ is the dual map of the map $d_{n+1}:C_{n+1}\to C_{n}$, and so we can write $C^\bullet$ as
$$\xymatrix{0&C^2\ar[l]&C^1\ar[l]_{d_2^*}^{d^1}&C^0\ar[l]_{d_1^*}^{d^0}&0\ar[l]},$$
with $C^n:=\mathrm{Hom}_\Z(C_n,\Z)$. Recall that the dual of a $\Z$-linear map $d:X\to Y$ is a linear map $d^*:Y^*\to X^*$ defined by $d^*(\phi):=\phi\circ d$ between the  dual spaces $Y^*:=\mathrm{Hom}_\Z(Y,\Z)$ and $X^*:=\mathrm{Hom}_\Z(X,\Z)$. If the linear map $d$ is represented by a matrix $A$ with respect to two bases of $X$ and $Y$, then the linear map $d^*$ is represented by the transpose matrix $A^T$ with respect to the dual bases of $Y^*$ and $X^*$. This can be seen from the following diagram
$$\xymatrix{\Z^a\ar@/^/[r]^{A}\ar[d]_f^{\wr}&\Z^b\ar[l]^{A^T}\ar[d]^g_{\wr}\\
\mathrm{Hom}_{\Z}(\Z^a,\Z)&\mathrm{Hom}_{\Z}(\Z^b,\Z)\ar[l]^{A^*},
}$$
$$fA^Tg^{-1}\phi=fA^T\phi(I)^T=(\vx\mapsto \phi(I)A\vx)=\phi\circ A=A^*\phi,$$
where $g(\vv):=\vx\mapsto \vv^T\vx$,\,\, $g^{-1}(\phi):=\phi(I)^T$,\,\, $\phi(I):=(\phi(\ve_1),\ldots(\ve_b))$, with $\ve_i$ being the columns of the $b\times b$ identity matrix, and the map $f$ is defined similarly.
Notice that for any $\phi\in C^n$, $$d^{n+1}\circ d^{n}=d_{n+2}^*\circ d_{n+1}^*(\phi)=\phi\circ d_{n+1}\circ d_{n+2}=0.$$
We define the $i$-th cohomology group $H^i(C^\bullet):=Z^i(C^\bullet)/B^i(C^\bullet)$, and the sequence of  groups $H^\bullet(C^\bullet):=\{H^i(C^\bullet)\}_{i\in\N_0}$.
The cellular cohomology for the CW-complex $\Gamma$ is defined as $H^\bullet(\Gamma):=H^\bullet(C^\bullet)$.
The cohomology groups are given by
\begin{eqnarray*}
  H^2(\Gamma)&=&C^2/ \text{Im}\, d_2^*=\text{coker}\, d_2^*\\
  H^1(\Gamma)&=&\ker d_2^*/ \text{Im}\, d_1^*\\
  H^0(\Gamma)&=&\ker d_1^*.
\end{eqnarray*}

 If $\gamma:\Gamma\to\Gamma$ is a cellular map from the finite CW-complex to itself, then
 by homology theory it induces a map $\gamma^\bullet:C^\bullet\to C^\bullet$ between the cochain complex $C^\bullet$ of $\Gamma$ and itself.
Let $\gamma^\bullet:=\{\gamma_n:C^n\to C^n\}_{n\in\N_0}$, i.e.
\begin{equation}\label{eq:gammadot}
\xymatrix{0\ar[r]&C^0\ar[r]^{d^0}\ar[d]_{\gamma_0}&C^1\ar[r]^{d^1}\ar[d]_{\gamma_1}&C^2\ar[r]\ar[d]_{\gamma_2}&0\\
            0\ar[r]&C^0\ar[r]^{d^0}&C^1\ar[r]^{d^1}&C^2\ar[r]&0}
\end{equation}
 be such commutative diagram. Since the diagram commutes, if $x\in \ker d^i$ then $\gamma_i(x)\in \ker d^i$, as $d^i(x)=0$ implies that $0=\gamma_{i+1}d^i(x)=d^i\gamma_i(x)$.
 The cochain map $\gamma^\bullet$ induces a sequence of maps $\{H^i(\gamma^\bullet):H^i(C^\bullet)\to H^i(C^\bullet)\}_{i\in\N_0}$,
where $H^i(\gamma^\bullet)$ is given by $[x]_{\im d^{i-1}}\mapsto [\gamma_i(x)]_{\im d^{i-1}}$.
An alternative notation of $H^i(C^\bullet)$ is $H^i(\Gamma,\Z)$.\\

\section{How to compute cohomology groups.}
We remark that  $\Z/n:=\Z/n\Z$, $n\in\N_0$, where $\Z/1=\Z/\Z=\{0\}$ and $\Z/0=\Z/0\Z=\Z$.
\begin{lem}\label{l:cokernel}
Suppose that the  matrix $A$ of a linear map $f:\Z^a\to\Z^b$ is an integer matrix of dimension $b\times a$.
Let $D:=diag(n_1,\ldots,n_{min(a,b)})$ be the Smith normal form of matrix $A$, and hence there exist invertible matrices $P,Q$ such that $D=PAQ$.
Then
$$\text{Coker } f=\frac{\Z^b}{A \Z^a}\cong\left\{
\begin{array}{l}
\frac{\Z}{n_1}\oplus\cdots\oplus\frac{\Z}{n_b} \qquad\qquad\,\,\,\,\text{ if }  b\le a\\
\frac{\Z}{n_1}\oplus\cdots\oplus\frac{\Z}{n_a}\oplus\Z^{b-a} \quad\, \text{ if } b>a.\\
\end{array}
\right.$$
The isomorphisms are given by
\begin{eqnarray*}
[x]_{A\Z^a} & \mapsto & ([(Px)_1]_{n_1},\ldots,[(Px)_b]_{n_b})\\
\, [x]_{A\Z^a} & \mapsto & ([(Px)_1]_{n_1},\ldots,[(Px)_a]_{n_a},(Px)_{a+1},\ldots,(Px)_{b})\\
([y_1]_{n_1},\ldots,[y_b]_{n_b}) & \mapsto &[P^{-1}y]_{A\Z^a} \\
([y_1]_{n_1},\ldots,[y_a]_{n_a},y_{a+1},\ldots,y_{b}) & \mapsto &  [P^{-1}y]_{A\Z^a},
\end{eqnarray*}
where the subindexes mean coordinate entries, i.e. $x=(x_1,\ldots,x_b)$.
\end{lem}

\begin{proof}
Since the matrix $A$  is an integer matrix, it has a Smith normal form. That is, we can write $A$ as a product of three matrices $P^{-1}DQ^{-1}$,
 where $P$ is a change of basis in $\Z^b$,  $Q$ is a change of basis of $\Z^a$, and $D=diag(n_1,\ldots,n_r,0,\ldots,0)$ is a diagonal matrix with nonzero integers in the first $r$ entries of the diagonal, and $n_i$ divides $n_{i+1}$.
Hence $P$ is a $b\times b$ matrix with determinant $\pm 1$, and $Q$ is a $a\times a$ matrix with determinant $\pm 1$.
Suppose that $b\le a$.
Define the map $\phi:\Z^b\to \frac{\Z}{n_1}\oplus\cdots\oplus\frac{\Z}{n_b}$ by
  $$\phi(x):=([(Px)_1]_{n_1},\ldots,[(Px)_b]_{n_b}),$$
  where $Px=((Px)_1,\ldots,(Px)_b)$. Its kernel is
  \begin{eqnarray*}
  \ker \phi&=&\{x\in \Z^b\mid (Px)_1=n_1y_1,\ldots, (Px)_b=n_by_b\text{ for any } y\in \Z^a\}\\
  &=&\{x\in \Z^b\mid Px=Dy\text{ for any } y\in \Z^a\}\\
  &=&\{x\in \Z^b\mid Px=DQ^{-1}z\text{ for any } z\in \Z^a\}\\
  &=&P^{-1}DQ^{-1}\Z^a=A\Z^a.
  \end{eqnarray*}
  Since $\phi$ is surjective, $\frac{\Z^b}{A\Z^a}=\frac{\Z^b}{\ker \phi}\cong\text{Im} \phi= \frac{\Z}{n_1}\oplus \cdots\oplus\frac{\Z}{n_b}$, where the isomorphism is given by
  $[x]_{A\Z^a}  \mapsto \phi(x)= ([(Px)_1]_{n_1},\ldots,[(Px)_b]_{n_b})$.
To find the inverse, let $([y_1]_{n_1},\ldots,[y_b]_{n_b})=([(Px)_1]_{n_1},\ldots,[(Px)_b]_{n_b})$.
Then $(Px)_1=y_1+n_1z_1, \ldots, (Px)_b=y_b+n_bz_b$ for some $z\in\Z^a$. Hence $Px=y+Dz$. Hence $x=P^{-1}y+P^{-1}Dz=P^{-1}y+P^{-1}DQ^{-1}w=P^{-1}y+Aw$ for some $w\in\Z^a$. Thus, the inverse of $([y_1]_{n_1},\ldots,[y_b]_{n_b})$ is $[P^{-1}y+Aw]_{A\Z^a}=[P^{-1}y]_{A\Z^a}$.
This completes the first case.

Now suppose that $b>a$.
Define the map $\phi:\Z^b\to \frac{\Z}{n_1}\oplus\cdots\oplus\frac{\Z}{n_a}\oplus\Z^{b-a}$ by
  $$\phi(x):=([(Px)_1]_{n_1},\ldots,[(Px)_a]_{n_a},(Px)_{a+1},\ldots,(Px)_{b}),$$
  where $Px=((Px)_1,\ldots,(Px)_b)$. Its kernel is
  \begin{eqnarray*}
  \ker \phi&=&\{x\in \Z^b\mid (Px)_1=n_1y_1,\ldots, (Px)_a=n_ay_a,\\
  &&\qquad\qquad (Px)_{a+1}=0,\ldots, (Px)_b=0\text{ for any } y\in \Z^a\}\\
  &=&\{x\in \Z^b\mid Px=Dy\text{ for any } y\in \Z^a\}\\
  &=&\{x\in \Z^b\mid Px=DQ^{-1}z\text{ for any } z\in \Z^a\}\\
  &=&P^{-1}DQ^{-1}\Z^a=A\Z^a.
  \end{eqnarray*}
  Since $\phi$ is surjective, $\frac{\Z^b}{A\Z^a}=\frac{\Z^b}{\ker \phi}\cong\text{Im} \phi= \frac{\Z}{n_1}\oplus \cdots\oplus\frac{\Z}{n_a}\oplus\Z^{b-a}.$
  To find the inverse, let $$([y_1]_{n_1},\ldots,[y_a]_{n_a},y_{a+1},\ldots,y_b)=([(Px)_1]_{n_1},\ldots,[(Px)_a]_{n_a},(Px)_{a+1},\ldots,(Px)_b).$$
Then $(Px)_1=y_1+n_1z_1, \ldots, (Px)_a=y_a+n_az_a$, $(Px)_{a+1}=y_{a+1}$, $(Px)_b=y_b$ for some $z\in\Z^a$. Hence $Px=y+Dz$. Hence $x=P^{-1}y+P^{-1}Dz=P^{-1}y+P^{-1}DQ^{-1}w=P^{-1}y+Aw$ for some $w\in\Z^a$. Thus, the inverse of $([y_1]_{n_1},\ldots,[y_a]_{n_a},y_{a+1},\ldots,y_b)$ is $[P^{-1}y+Aw]_{A\Z^a}=[P^{-1}y]_{A\Z^a}$.
This completes the  second case.
\end{proof}
For reference purposes, we rewrite the above lemma using terminology from the section on cellular cohomology introduced earlier.
This lemma tells us how to compute the second cohomology group of the finite CW-complex $\Gamma$.
\begin{lem}\label{l:secondcohomology}
Suppose that the  matrix $A^1$ of the boundary map $d^1:\Z^E\to\Z^F$ is an integer matrix of dimension $F\times E$.
Let $D:=diag(n_1,\ldots,n_{min(E,F)})$ be the Smith normal form of matrix $A^1$, and hence there exist invertible matrices $P,Q$ such that $D=PA^1Q$.
Then the second cohomology group  $H^2(\Gamma)$ is

$$H^2(\Gamma)=\text{Coker } d^1=\frac{\Z^F}{A^1 \Z^E}\cong\left\{
\begin{array}{l}
\frac{\Z}{n_1}\oplus\cdots\oplus\frac{\Z}{n_F} \qquad\qquad\,\,\,\,\text{ if }  F\le E\\
\frac{\Z}{n_1}\oplus\cdots\oplus\frac{\Z}{n_E}\oplus\Z^{F-E} \quad\, \text{ if } F>E.\\
\end{array}
\right.$$
The isomorphisms are given by
\begin{eqnarray*}
[x]_{A^1\Z^E} & \mapsto & ([(Px)_1]_{n_1},\ldots,[(Px)_F]_{n_F}),\\
\, [x]_{A^1\Z^E} & \mapsto & ([(Px)_1]_{n_1},\ldots,[(Px)_E]_{n_E},(Px)_{E+1},\ldots,(Px)_{F}),
\end{eqnarray*}
\begin{eqnarray*}
([y_1]_{n_1},\ldots,[y_F]_{n_F}) & \mapsto &[P^{-1}y]_{A^1\Z^E}, \\
([y_1]_{n_1},\ldots,[y_E]_{n_E},y_{E+1},\ldots,y_{F}) & \mapsto &  [P^{-1}y]_{A^1\Z^E},
\end{eqnarray*}
where the subindexes mean coordinate entries, i.e. $x=(x_1,\ldots,x_F)$.
\end{lem}

\begin{lem}\label{l:isomchainisomhomology}
  Suppose that $\phi_{\bullet}:C_{\bullet}\tilde{\longrightarrow} E_\bullet$ is an isomorphism between two chain complexes of $\Z$-Modules.
  Then their homology groups  are isomorphic, i.e. $H_i(C_\bullet)\cong H_i(E_\bullet)$.
  The ith-isomorphism  is given by
  \begin{eqnarray*}
    [x]_{B_i(C_\bullet)}&\mapsto&[\phi_i(x)]_{B_i(E_\bullet)}\\
   \, [y]_{B_i(E_\bullet)}&\mapsto& [\phi_i^{-1}(y)]_{B_i(C_\bullet)}.
  \end{eqnarray*}
  The same holds for cochains.
\end{lem}
\begin{proof}
  By assumption we have the following commutative diagram
$$\xymatrix{\ar[r]&C_{i+1}\ar[r]^{d_{i+1}}\ar[d]_{\phi_{i+1}}^{\wr}&C_i\ar[r]^{d_i}\ar[d]_{\phi_i}^{\wr}&C_{i-1}\ar[r]\ar[d]_{\phi_{i-1}}^{\wr}&\\
            \ar[r]&E_{i+1}\ar[r]^{\partial_{i+1}}&E_i\ar[r]^{\partial_i}&E_{i-1}\ar[r]&.}$$
We would like to show that $\frac{\ker d_i}{\text{Im } d_{i+1}}\cong \frac{\ker \partial_i}{\text{Im } \partial_{i+1}}$.
We have $\ker d_i\cong \ker \partial_i$ and $\text{Im } d_{i+1}\cong\text{Im }\partial_{i+1}$ as
$x\in\ker d_i\iff \phi_i(x)\in\ker\partial_i$ and $y\in\text{Im }d_{i+1}\iff \phi_{i}(y)\in\text{Im }\partial_{i+1}$.
These isomorphisms of the kernel and image are a restriction of $\phi_i$.
Since the quotient map $q_i:\ker\partial_i\to \ker\partial_i/\im\partial_{i+1}$ is continuous, the map
$q_i\circ\phi_i:\ker d_i:\to \ker \partial_i/\im\partial_{i+1}$ given by $x\mapsto[\phi_i(x)]_{\text{Im }\partial_{i+1}}$ is continuous.
If $x,x'\in\ker d_i$ are equivalent, then $x=x'+z$ for some $z\in \text{Im }d_{i+1}$.
Then $[\phi_i(x)]_{\text{Im }\partial_{i+1}}=[\phi_i(x')+\phi_i(z)]_{\text{Im }\partial_{i+1}}=[\phi_i(x')]_{\text{Im }\partial_{i+1}}$ because
$\phi_i(z)\in\text{Im } \partial_{i+1}$.
Hence $q_i\circ\phi_i$ descends to the quotient, and so there is a unique continuous map
 $\psi_i: \frac{\ker d_i}{\text{Im } d_{i+1}}\to \frac{\ker \partial_i}{\text{Im } \partial_{i+1}}$ given by $\psi_i([x]_{\text{Im } d_{i+1}}):=[\phi_i(x)]_{\text{Im } \partial_{i+1}}$.
 Similarly, using instead $\phi^{-1}$, there is  a unique continuous map
  $\psi_i^{-1}: \frac{\ker \partial_i}{\text{Im } \partial_{i+1}}\to \frac{\ker d_i}{\text{Im } d_{i+1}}$ given by $\psi_i^{-1}([x]_{\text{Im } \partial_{i+1}}):=[\phi_i^{-1}(x)]_{\text{Im } d_{i+1}}$, which is the inverse of $\psi_i$.

\end{proof}

Continuing using the terminology of the section on cellular cohomology, the next lemma tells us how to compute the first cohomology group $H^1(\Gamma)$.
\begin{lem}\label{l:cohomology1}
  Suppose that the matrix $A^1$ of the boundary map $d^1:\Z^E\to \Z^F$ is an integer matrix.
  Let $D:=diag(n_1,\ldots,n_r,0,\ldots,0)$ be the Smith normal form of $A^1$, and hence there exist invertible matrices $P,Q$ such that $D=PA^1Q$, and $n_i\ne0$.
  Let $e_i$ be columns of the $E\times E$ identity matrix.
  Let $J:=[e_{r+1}\,\,\cdots\,\,e_E]$ be the matrix containing $E-r$ columns.
  Then the 1-cohomology group is
  \begin{eqnarray*}
  H^1(\Gamma)&=&  \frac{\ker d^1}{\text {Im } d^0}\cong \frac{\Z^{E-r}}{J^TQ^{-1}d^0\Z^V}\cong\\\\
  &\cong&\left\{
  \begin{array}{l}
    \frac{\Z}{m_1}\oplus\cdots\oplus\frac{\Z}{m_{E-r}} \qquad\qquad\,\,\,\,\text{ if } E-r-V\le0\\
    \frac{\Z}{m_1}\oplus\cdots\oplus\frac{\Z}{m_V}\oplus\Z^{E-r-V} \,\,\,\,\,\text{ if } E-r-V> 0,\\
   \end{array}
   \right.
     \end{eqnarray*}
  where  $\tilde D:=diag(m_1,\ldots,m_{min(E-r,V)})$ is the Smith normal form of matrix $J^TQ^{-1}d^0$,  and hence there exists invertible matrices $\tilde P,\tilde Q$ such that $\tilde D=\tilde P J^TQ^{-1}d^0 \tilde Q$.
  The isomorphisms for the second case (when $E-r-V>0$) are
\begin{scriptsize}
\begin{eqnarray*}
  [x]_{\im d^0}\mapsto
  ([(\tilde P J^T Q^{-1} x)_1]_{m_1},\ldots,[(\tilde P J^T Q^{-1} x)_V]_{m_V},(\tilde P J^T Q^{-1} x)_{V+1},(\tilde P J^T Q^{-1} x)_{E-r})\\
  ([y_1]_{m_1},\ldots,[y_V]_{m_V},y_{V+1},\ldots,y_{E-r})\mapsto [QJ\tilde P ^{-1} y]_{\im d^0}.\qquad\qquad\qquad\qquad\qquad\qquad\quad
\end{eqnarray*}
\end{scriptsize}

\end{lem}
\begin{proof}
  We have the following commutative diagram of cochain complexes:
$$\xymatrix{0\ar[r]&\Z^V\ar[r]^{d^0}\ar[d]_{id}&\Z^E\,\,\,\,\,\ar[r]^{P^{-1}DQ^{-1}}\ar[d]_{Q^{-1}}&\,\,\,\Z^F\ar[r]\ar[d]_{P}&0\\
            0\ar[r]&\Z^V\ar[r]^{Q^{-1}d^0}&Q^{-1}\Z^E\ar[r]^{D}&\Z^F\ar[r]&0}$$
By the previous Lemma \ref{l:isomchainisomhomology} $\frac{\ker d^1  }{\im d^0}=\frac{\ker P^{-1}DQ^{-1}  }{\im d^0}\cong \frac{\ker D}{\im Q^{-1}d^0}$, with $Q^{-1}$ inducing the isomorphism.
Since $D$ is a diagonal matrix, $\ker D=\{x\in Q^{-1}\Z^E\mid n_1x_1=0,\cdots,n_rx_r=0\}=0^r\oplus \Z^{E-r}$.
Since the image of $Q^{-1}d^0$ is contained in $\ker D$, we can remove the first $r$-entries of each vector in $\im Q^{-1}d^0$.
This can be accomplished by multiplying the matrix $J^T$ with $Q^{-1}d^0$, that is $J^T$ removes the first $r$ rows of the matrix $Q^{-1}d^0$. Hence $\frac{\ker D}{\im Q^{-1}d^0}=\frac{0^r\oplus \Z^{E-r}}{\im Q^{-1}d^0}\cong \frac{\Z^{E-r}}{J^TQ^{-1}d^0\Z^V}=\text{coker }J^TQ^{-1}d^0.$
The rest follows by Lemma \ref{l:cokernel} with $A=J^{T}Q^{-1}d^0$, $a=V$ and $b=E-r$.
\end{proof}

\begin{lem}\label{l:AZntoZr}
Let $A$ be a $n\times n$ integer matrix. Let $D:=\mathrm{diag}(\lambda_1,\ldots,\lambda_r,0,\ldots,0)$, be the Smith normal form of matrix $A$, and hence there exist $\Z$-invertible matrices $B,C$ such that $D=C^{-1}AB$, where $\lambda_i\ne0$ for all $i\le r$ for some $r\le n$. Let $\ve_i$, $i=1,\ldots,n$ be the columns of the $n\times n$ identity matrix. Let $\imath=[\ve_1,\ldots,\ve_r]$ be the inclusion map $\Z^r\to\Z^n$. Let $D':=\mathrm{diag}(\lambda_1,\ldots,\lambda_r)$ be the $r\times r$ diagonal matrix. Then $$A\Z^n\cong \Z^r,$$ where the isomorphisms are $D'^{-1}\imath^TC$, $ C \imath D'$.
\end{lem}
\begin{proof}
 Let $P:=D'^{-1}\imath^TC^{-1}$ and $Q:=C \imath D'$. Then $P$ maps $A\Z^n$ into $\Z^r$, and $Q$ maps $\Z^r$ into $A\Z^n$.
 Since $QPA=A$ and $PQ=I$, the image $A\Z^n$ is isomorphic to $\Z^r$, with $P,Q$ being the isomorphisms.
\end{proof}
%\newpage
\section{Direct Limits}
In this section we consider direct limits of sequences of $\Z$-modules and $\Z$-linear maps.
Although such direct limits can be defined abstractly in terms of a universal property, we choose the concrete definition instead.
We should remark that a $\Z$-module can be considered as an Abelian group and viceversa.

Suppose
$$\xymatrix{Z_1\ar[r]^{A_{1}}&Z_2\ar[r]^{A_{2}}&Z_3\ar[r]^{A_{3}}&Z_4\ar[r]^{A^4}&\cdots}$$
is a sequence of $\Z$-modules $Z_i$ and $\Z$-linear maps $A_{i}:A_i\to A_{i+1}$. The\emph{ direct limit} of such sequence is defined as
$$\lim_{\rightarrow}(A_{i},Z_i):=\left(\bigoplus_{i\in\N} Z_i\right)/\sim,$$
where the equivalence relation $\sim$ is defined as follows:
Let $\imath_k:Z_k\to \bigoplus_{i\in\N} Z_i$ be the inclusion map $\imath_k(z_k):=(0,\ldots,0,z_k,0,\ldots)$, where $z_k$ is in the $k$-entry.
For $z_i\in Z_i$ and $z_j\in Z_j$ we say that $\imath_i(z_i)\sim \imath_j(z_j)$ if there is some integer $k\ge i,j$ such that $A_{i,k}z_i=A_{j,k}z_j$, where $A_{i,k}:=A_{k-1}\circ\cdots\circ A_{i}$ maps $Z_i$ into $Z_k$.
For instance $(z_1,z_2,0,\ldots)=(z_1,0,0,\ldots)+(0,z_2,0,\ldots)\sim(0,A_1z_1,0,\ldots)+(0,z_2,0,\ldots)=(0,A_1z_1+z_2,0,\ldots)$.
Recall that any element $\vx \in \bigoplus_{i\in\N} Z_i$ is a sequence $\vx=(x_i)_{i\in\N}$, where $x_i\ne 0$ for a finite number of them, i.e. a a finite number of entries in $\vx$ are different from zero.
It follows that for $\vx,\vy\in  \bigoplus_{i\in\N} Z_i $ we have $\vx\sim \vy$ if there is an integer $k\in\N$ such that $x_i=y_i=0$ for $i>k$ and
$$\sum_{i=1}^k A_{i,k}x_i=\sum_{i=1}^k A_{i,k}y_i.$$
In particular, $\vx\sim (0,\ldots,0,z_k,0,\ldots)$ where $z_k$ is the $k$-entry,  $z_k:=\sum_{i=1}^k A_{i,k}x_i$, and $z_i=0$ for $i>k$, so any element in the direct limit can be written as $[(0,\ldots,0,z_k,0,\ldots)]_\sim$, for some integer $k\in\N$ and $z_k\in Z_k$.
Observe that we have canonical $\Z$-linear maps $q_k:Z_k\to \lim \limits_{\rightarrow}(A_i,Z_i)$ given by $q_k(z_k):= [\iota_k(z_k)]_\sim$.\\

We are interested in finding the direct limit, denoted by $\lim\limits_{\rightarrow}(A,\Z^n)$, of the sequence
$$\xymatrix{\Z^n\ar[r]^{A}&\Z^n\ar[r]^{A}&\Z^n\ar[r]^{A}&\Z^n\ar[r]^{A}&\cdots},$$
where $A$ is an $n\times n$ integer matrix. Solving this problem is by no means a trivial task.
If $A$ is an $n\times n$ integer matrix with nonzero determinant then
$$\lim\limits_{\rightarrow}(A,\Z^n)\cong\bigcup_{k\in\N} A^{-k}\Z^n\subset (\Z[\tfrac 1{\det A}])^n,$$
which follows from the following commutative diagram
  $$\xymatrix{\Z^n\ar[r]^{A}\ar[d]_{I}&\Z^n\ar[r]^{A}\ar[d]_{A^{-1}}&\Z^n\ar[r]^{A}\ar[d]_{A^{-2}}&\Z^n\ar[r]^{A}\ar[d]_{A^{-3}}&\\
  \Z^n\ar[r]_{\imath_0}&A^{-1}\Z^n\ar[r]_{\imath_1}&A^{-2}\Z^n\ar[r]_{\imath_2}&A^{-3}\Z^n\ar[r]_{\imath_3}&,\\
 }$$
 where $\imath_i(\vx):=\vx$ is the inclusion map and $(\det A)^iA^{-i}\Z^n\subset\Z^n$, $\Z[\tfrac1d]:=\cup_{i\in\N_0}\frac1{d^i}\Z$.
If the determinant is zero then we can use Lemma \ref{l:AZntoZr} on the  integer matrix $A$ to obtain a smaller $m\times m$ integer matrix $A'$ with nonzero determinant and
$$\lim\limits_{\rightarrow}(A,\Z^n)\cong\lim\limits_{\rightarrow}(A,A\Z^n)\cong\lim\limits_{\rightarrow}(A',\Z^m).$$

The following lemma  shows how to $\Z$-conjugate some integer matrices to integer matrices whose first column is parallel with $\ve_1$.

\begin{lem}\label{l:generalizeddirectlimitGinvAG}
  Suppose $A$ is an $n\times n$ integer matrix. Suppose $\lambda$ is an eigenvalue of $A$ with corresponding integer eigenvector $\vx$.
  Let $\ve_1,\ldots,\ve_n$ be the columns of the $n\times n$ identity matrix $I$.
  Let $G:=[\vx,e_2,\ldots,e_n]$. If the first entry of $\vx$ is 1, then $G$ is $\Z$-invertible and $G^{-1}AG e_1= \lambda e_1$.
\end{lem}
\begin{proof}
  Since $G=[\vx,e_2,\ldots,e_n]=[\vx-\ve_1,0,\ldots,0]+I$ is a lower triangular matrix with integer entries and ones on the diagonal, its determinant is $1$ hence $\Z$-invertible.
  Moreover $G^{-1}=[2\ve_1-\vx,\ve_2,\ldots,\ve_n]=[\ve_1-\vx,0,\ldots,0]+I.$
  Then, $G^{-1}AG \ve_1=G^{-1}A\vx=\lambda G^{-1}\vx=\lambda(\ve_1-\vx+\vx)=\lambda \ve_1$.
\end{proof}

If the same lemma \ref{l:generalizeddirectlimitGinvAG} holds for the lower right $(n-1)\times (n-1)$ block of $G^{-1}AG$, then $A$ is $\Z$-conjugate to an integer matrix that has zeroes below the diagonal on the first and second column. In this way we can continue putting zeroes below the diagonal until the above lemma \ref{l:generalizeddirectlimitGinvAG} no longer holds or until we get an upper triangular matrix.

\begin{lem}\label{l:directlimitGinvAG}
  Suppose $U$ is an integer $n\times n$ upper triangular matrix with nonzero diagonal entries $\lambda_1,\ldots,\lambda_n$.
  Let $D:=\mathrm{diag}(\lambda_1,\ldots,\lambda_n)$ be the invertible diagonal matrix with same diagonal entries as $U$.\\
  (a) If $V_k:=U^kD^{-k}$ is an integer matrix for all $k\in \N$  then
  $$\lim\limits_{\rightarrow} (U,\Z^n)\cong\lim\limits_{\rightarrow} (D,\Z^n)]\cong\Z[\tfrac1{\lambda_1}]\oplus\cdots\oplus\Z[\tfrac1{\lambda_n}].$$
  (b) If  $\lambda_2=\cdots=\lambda_n=1$ then
  $$\lim\limits_{\rightarrow} (U,\Z^n)\cong\lim\limits_{\rightarrow} (D,\Z^n)\cong\Z[\tfrac1{\lambda_1}]\oplus\Z^{n-1}.$$
\end{lem}
\begin{proof}
   Notice that matrix $V_k$ is upper triangular as matrices $U$ and $D$ are.
  The $ij$ entry of matrix $V_k$ is $(V_k)_{ij}=\sum_{m=1}^n (U^k)_{im}\frac{\delta_{mj}}{\lambda_j^k}=(U^k)_{ij}\frac{\delta_{ij}}{\lambda_j^k}$.
  Since the diagonal entries in $U^k$ are precisely $\lambda_j^k$, $j=1,\ldots,n$, the diagonal entries of $V_k$ are $(V_k)_{jj}=1$.
  Since $V_k$ is assumed to be an integer matrix and since it is upper triangular its determinant 1, hence $\Z$-invertible.
  Since the following diagram commutes
  $$\xymatrix{\Z^n\ar[r]^{D}\ar[d]_{I}&\Z^n\ar[r]^{D}\ar[d]_{V_1}&\Z^n\ar[r]^{D}\ar[d]_{V_2}&\Z^n\ar[r]^{D}\ar[d]_{V_3}&\\
  \Z^n\ar[r]^{U}&\Z^n\ar[r]^{U}&\Z^n\ar[r]^{U}&\Z^n\ar[r]^{U}&\\}$$
  and since  $V_k$ is $\Z$-invertible, the first part of the lemma follows.
  In the case when $\lambda_2=\cdots=\lambda_n=1$, the $ij$ entry of $V_k$ is $(V_k)_{ij}=(U^k)_{ij}\delta_{ij}$ for $j=2,\ldots,n$, and $(V_k)_{jj}=1$. Hence $V_k$ is an integer matrix, and by the first part of the lemma, the second part follows.
\end{proof}
For later use in our determination of a cohomology group, we apply the previous remarks to the following example.

\begin{exam}\label{ex:fivebyfivemainmatrix}
  Defining the following matrices
  $$A:=\left(
\begin{array}{ccccc}
 2 & 2 & 2 & 1 & 1 \\
 0 & 1 & 0 & 0 & 0 \\
 1 & 2 & 3 & 1 & 1 \\
 1 & 2 & 2 & 2 & 1 \\
 1 & 2 & 2 & 1 & 2 \\
\end{array}
\right)
\qquad
G:=\left(
\begin{array}{ccccc}
 1 & 0 & 0 & 0 & 0 \\
 0 & 1 & 0 & 0 & 0 \\
 1 & 0 & 1 & 0 & 0 \\
 1 & 0 & 0 & 1 & 0 \\
 1 & 0 & 0 & 0 & 1 \\
\end{array}
\right),
$$
where we used Lemma \ref{l:generalizeddirectlimitGinvAG} to construct the above matrix $G$, we get
$$G^{-1}AG=\left(
\begin{array}{ccccc}
 6 & 2 & 2 & 1 & 1 \\
 0 & 1 & 0 & 0 & 0 \\
 0 & 0 & 1 & 0 & 0 \\
 0 & 0 & 0 & 1 & 0 \\
 0 & 0 & 0 & 0 & 1 \\
\end{array}
\right),$$
 which is an upper triangular matrix. By the second Lemma \ref{l:directlimitGinvAG} we have $$\lim\limits_{\rightarrow}(A,\Z)\cong\lim\limits_{\rightarrow}(\mathrm{diag}(6,1,1,1,1),\Z)\cong\Z[\tfrac 16]\oplus\Z^4.$$
\end{exam}
%\newpage
\section{Computing the cohomology of the continuous space $\Omega$}
\subsection{Cohomology of the continuous space $\Omega$}\label{subs:cohomologyofOmega}
\begin{figure}
  \includegraphics[scale=.30]{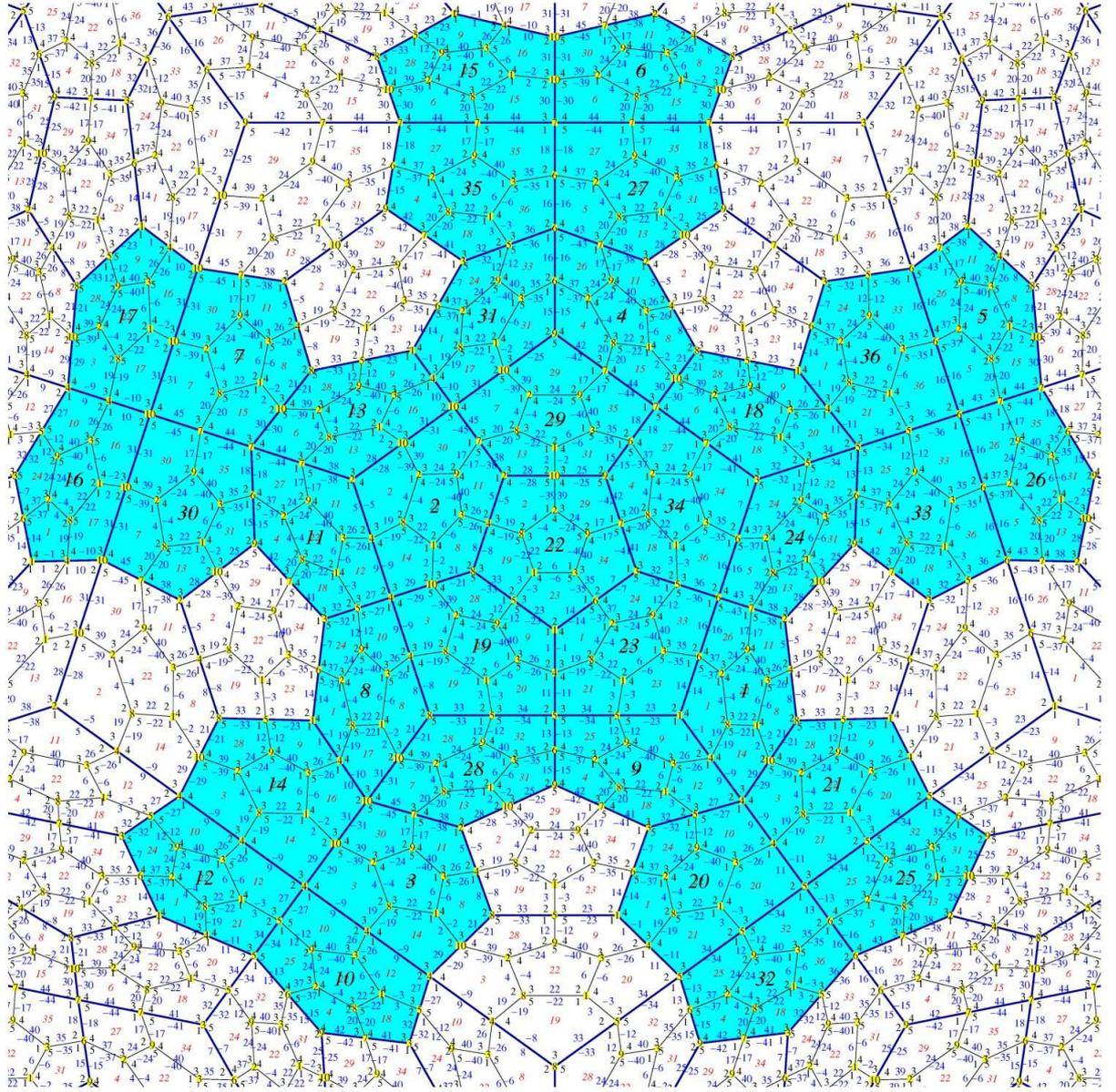}\\
  \caption{Collared tiles, edges, vertices, subdivision of them, and decoration. See \cite{MRSinverselimit} for explanation.}\label{f:decoratedtileseverything}
\end{figure}
In \cite{MRSinverselimit} we constructed a collared substitution $\omega:\Omega\to\Omega$ on the continuous hull $\Omega$, which is a homeomorphism.
Each collared tile is a regular pentagon of side-length 1 plus the degree of its vertices and inside and outside decoration of the pentagon.
There are 36 collared tiles, 45 collared edges and 10 collared vertices. Figure \ref{f:decoratedtileseverything} shows them and their collared substitution.
In the same article \cite{MRSinverselimit} we showed that $\Omega$ induces a finite CW-complex $\Gamma$, which we call the Anderson-Putnam finite CW-complex, where the closed 2-cells, closed 1-cells, and 0-cells are precisely the collared tiles, collared edges and collared vertices.
Moreover, we saw that the map $\omega$ induces a continuous surjective map $\gamma:\Gamma\to\Gamma$, and that $\Omega$ is homeomorphic to the inverse limit $\Omega_1:=\lim\limits_{\leftarrow}(\Gamma,\gamma)$. By the proof of Theorem 6.1 in \cite{PutnamBible95}, the Cech cohomology with integer coefficients $H^i(\Omega,\Z)$  is isomorphic to the direct limit of the system of Abelian groups
\begin{equation*}
\xymatrix{H^i(\Gamma,\Z)\ar[r]^{H^i(\gamma^\bullet)}&H^i(\Gamma,\Z)\ar[r]^{H^i(\gamma^\bullet)}&H^i(\Gamma,\Z)\ar[r]\ar[r]^{\quad H^i(\gamma^\bullet)}&\cdots},
\end{equation*}
for $i=0,1,2,\ldots$. That is, for fixed $i\in\N_0$
\begin{equation}\label{e:cohomologyasdirlim}
 H^i(\Omega,\Z)\cong \lim_{\to}(H^i(\Gamma,\Z),H^i(\gamma^\bullet)).
\end{equation}
Moreover, by Theorem 6.3 in \cite{PutnamBible95}, the K-theory of the continuous hull $\Omega$ is
\begin{eqnarray}\label{e:ktheory}
 && K^0(\Omega)\cong H^0(\Omega,\Z)\oplus H^2(\Omega,\Z)\\
\nonumber && K^1(\Omega)\cong H^1(\Omega,\Z).
\end{eqnarray}

The cochain map $\gamma^\bullet:=\{\gamma_n:C^n\to C^n\}_{n\in\N_0}$ is given as follows.
By the degree of maps, the induced map $\gamma_2$ maps each collared tile $t$ to the sum of the six collared tiles of $\omega(t)$, as the collared tiles are all clockwise oriented. Similarly, by degree of maps, the induced map $\gamma_1$ maps each collared edge to the sum of two collared edges (with $\pm$ to take care of orientation) as a collared edge divides into two collared edges. Also, the induced map $\gamma_0$ maps each collared vertex to itself as a collared vertex remains still when the subdivision of a collared tile is performed. Hence $\gamma_0$ is the identity map. For $n>2$, $\gamma_n=0$.

The cellular chain complex $C_\bullet$ is given as follows.
The collared tiles shown in Figure \ref{f:decoratedtileseverything} (and listed in Table 1 in \cite{MRSinverselimit}) are an ordered basis of $C_2$ if we provide them with the ordering of the list.
Then $C_2\cong \Z^{36}$ with collared tile $t_i$ being mapped to $e_i$, where $\{e_i\}$ is the standard basis of $\Z^{36}$.
Similarly, the collared edges and vertices shown in Figure \ref{f:decoratedtileseverything} (and listed in Table 2 and Table 3 in \cite{MRSinverselimit}, respectively), are an ordered basis of $C_1$ and $C_0$, respectively, and thus $C_1\cong \Z^{45}$ and $C_0\cong\Z^{10}$.

The cellular cochain complex $C^\bullet$ in terms of matrices is given as follows.
Let $A^i:=A_{i+1}^T$ be the matrix of the boundary maps $d^i$. The cochain complex of $\Gamma$ is
$$\xymatrix{0&\Z^{36}\ar[l]&\Z^{45}\ar[l]_{A_2^T}^{A^1}&\Z^{10}\ar[l]_{A_1^T}^{A^0}&0\ar[l]}.$$
Since $d_2:C_2\to C_1$ sends each collared tile to the sum of its boundary collared edges (with $\pm$ on each collared edge to preserve orientation), each row in the matrix  $A^1:=A_2^T$ represents a collared tile, and each entry on this row represents a collared edge of the boundary of the collared tile (with $\pm$ to indicate orientation).
Similarly, since $d_1:C_1\to C_0$ sends each collared edge to the difference of its collared vertices, each row in the matrix $A^0:=A_1^T$ represents a collared edge, and each entry on this row represents a collared vertex of the boundary of the collared edge (with $\pm$ to indicate orientation).
Let $B_1,B_2$ be the matrices for the inflation maps $\gamma_1$, $\gamma_2$ respectively.
These matrices are shown next and can be read from Figure \ref{f:decoratedtileseverything}.   \\
\begin{center}
  \includegraphics[scale=.55]{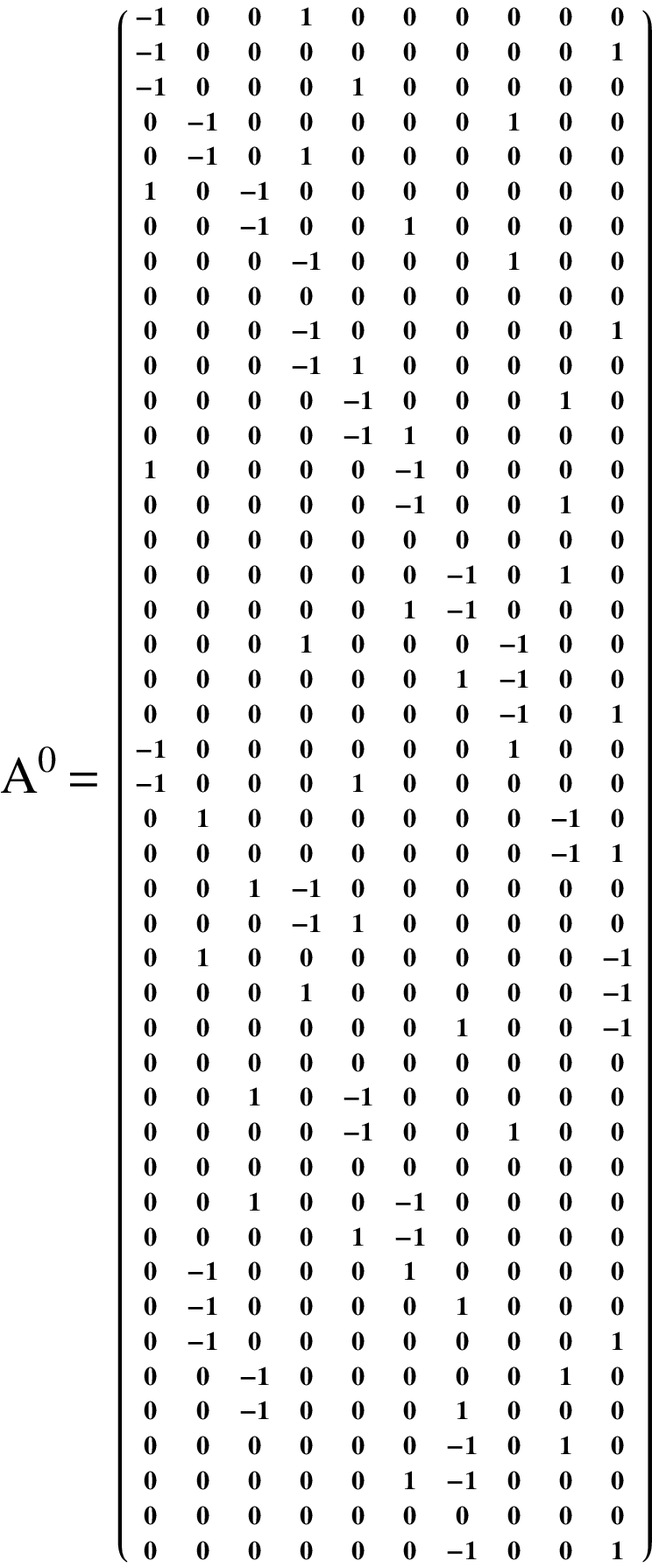}
\end{center}

\begin{center}
  \includegraphics[scale=.53]{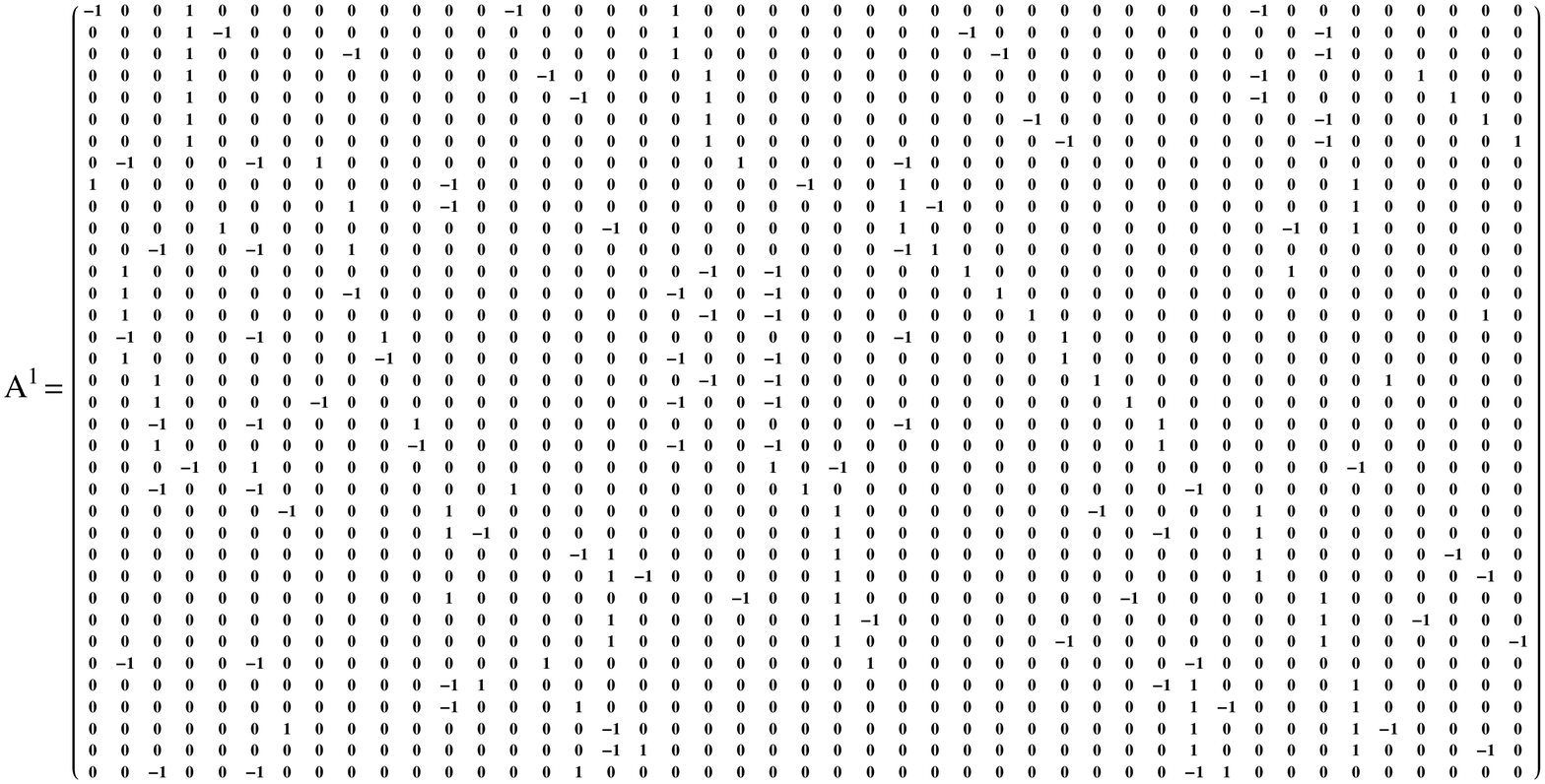}
\end{center}

\begin{center}
  \includegraphics[scale=.55]{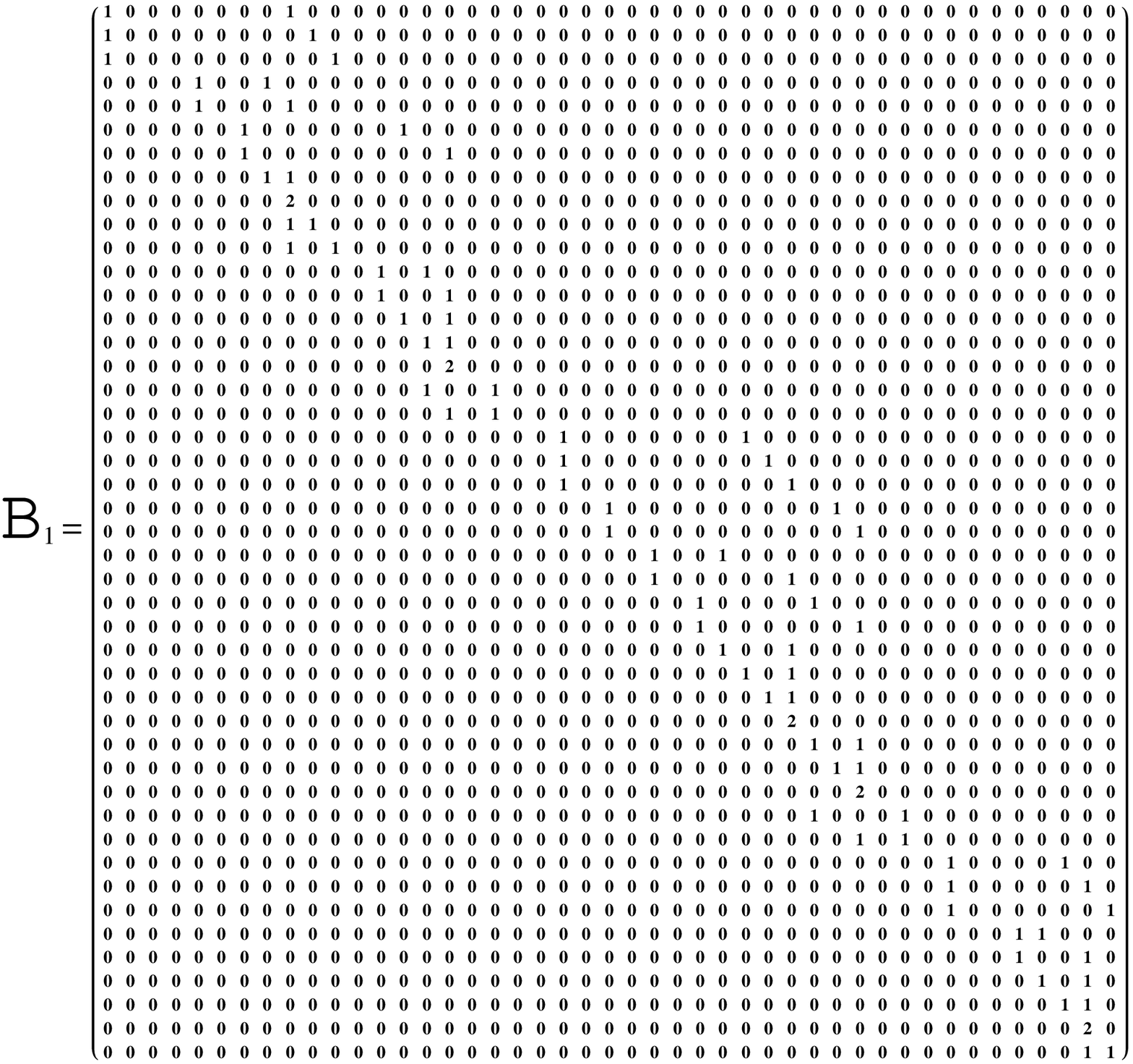}
\end{center}

\begin{center}
  \includegraphics[scale=.55]{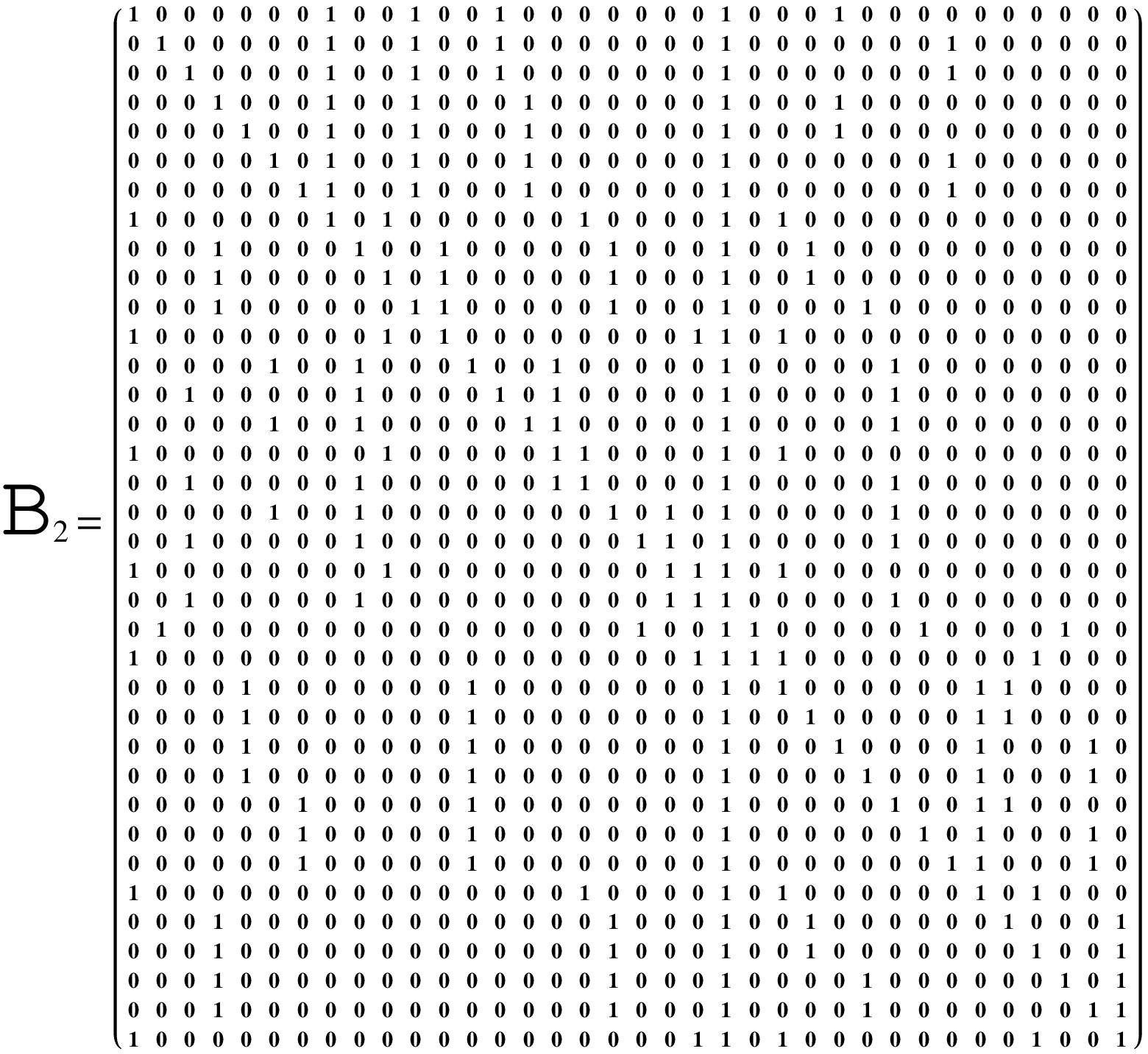}
\end{center}

The matrix $(B_2)^2$ is a positive matrix as every entry is at least 2. Hence  matrix $B_2$ is primitive, and the collared substitution $\omega:\Omega\to\Omega$ is  primitive.

Since $B_1A^0=A^0$ and $B_2A^1=A^1B_1$ the diagram in Equation (\ref{eq:gammadot}) commutes, as expected.
Since $\ker A^0\cong \Z$, the zero cohomology group $H^0(\Gamma)=\Z$.
Since $\gamma_0$ is the identity map, $H^0(\gamma^\bullet)$ is also the identity map, and therefore the zero-cohomology group of the continuous hull is
$$H^0(\Omega)=\lim\limits_{\rightarrow}(\Z,id)=\Z.$$

By Lemma \ref{l:cohomology1}, $H^1(\Gamma)=\frac{\ker A^1}{\im A^0}\cong\frac{\Z^{14}}{J^TQ^{-1}A^0\Z^{10}}\cong(\Z_1)^9\oplus\Z_0\oplus\Z^{4}\cong \Z^5$.
Using the same lemma, the map $H^1(\gamma^\bullet):H^1(\Gamma)\to H^1(\Gamma)$ is isomorphic to the map $(\Z_1)^9\oplus\Z_0\oplus\Z^{4}\to (\Z_1)^9\oplus\Z_0\oplus\Z^{4}$ given by
$$([x_1]_1,\ldots,[x_9]_1,[x_{10}]_0,x_{11},\ldots,x_{14})\mapsto ([y_1]_1,\ldots,[y_9]_1,[y_{10}]_0,y_{11},\ldots,y_{14}),$$
where $y=\tilde P J^T Q^{-1}\gamma_1 Q J\tilde P^{-1} x$.
This map turns out to be the zero map. Hence $H^1(\gamma^\bullet)=0$, and so the 1-cohomology group of $\Omega$ is
$$H^1(\Omega)=0.$$

By Lemma \ref{l:secondcohomology}, $H^2(\Gamma)=\text{Coker } d^1= \frac{\Z^{36}}{A^1\Z^{45}}\cong(\Z_2)^5\oplus \Z^{5}$.
Using the same lemma, the map $H^2(\gamma^\bullet):H^2(\Gamma)\to H^2(\Gamma)$ is isomorphic to the map
$(\Z_1)^{26}\oplus (\Z_2)^5\oplus\Z^{5}\to(\Z_1)^{26}\oplus (\Z_2)^5\oplus \Z^{5}$ given by
$x\mapsto P\gamma_2P^{-1}x$. Since $\Z_1=0$, this map is isomorphic to
$(\Z_2)^5\oplus\Z^{5}\to (\Z_2)^5\oplus \Z^{5}$ given by
$x\mapsto B_2'x$, where $B_2'$ is the lower right $10\times 10$ block of $P\gamma_2P^{-1}$. So
\begin{eqnarray*}
B_2':=\left(
\begin{array}{cccccccccc}
 0 & 0 & 0 & 0 & 0 & 0 & 0 & 0 & 0 & 0 \\
 0 & 0 & 0 & 0 & 0 & 1 & 1 & 1 & 0 & 0 \\
 0 & 0 & 0 & 0 & 0 & 1 & 1 & 1 & 0 & 0 \\
 0 & 0 & 0 & 0 & 0 & 0 & 1 & 1 & 1 & 0 \\
 0 & 0 & 0 & 0 & 0 & 0 & 0 & 0 & 0 & 0 \\
 0 & 0 & 0 & 0 & 0 & 2 & 2 & 2 & 1 & 1 \\
 0 & 0 & 0 & 0 & 0 & 0 & 1 & 0 & 0 & 0 \\
 0 & 0 & 0 & 0 & 0 & 1 & 2 & 3 & 1 & 1 \\
 0 & 0 & 0 & 0 & 0 & 1 & 2 & 2 & 2 & 1 \\
 0 & 0 & 0 & 0 & 0 & 1 & 2 & 2 & 1 & 2 \\
\end{array}
\right)=\left(
\begin{array}{cc}
 0 & M_1 \\
 0 & M_2 \\
\end{array}
\right),
\end{eqnarray*}
where $M_1$ is the upper right $5\times 5$ matrix and $M_2$ is the lower right $5\times 5$ matrix of $B_2'$.

\begin{lem}
The direct limit of
$$\xymatrix{&(\Z_2)^5\oplus \Z^{5}\ar[r]^{B_2'}&(\Z_2)^5\oplus \Z^{5}\ar[r]^{B_2'}&(\Z_2)^5\oplus \Z^{5}\ar[r]&}$$
is $\Z[\tfrac 16] \oplus\Z^4$.
\end{lem}
 \begin{proof}
Defining $M:=\left(
\begin{array}{c}
 M_1 \\
  M_2 \\
\end{array}
\right),$ and $\pi_2$ the projection map $(z_1,z_2)\mapsto z_2$, the following diagram commutes
   $$\xymatrix{(\Z_2)^5\oplus \Z^{5}\ar[r]^{B_2'}\ar[d]_{\pi_2}&(\Z_2)^5\oplus \Z^{5}\ar[r]^{B_2'}\ar[d]_{\pi_2}&(\Z_2)^5\oplus \Z^{5}\ar[r]^{B_2'}\ar[d]_{\pi_2}&(\Z_2)^5\oplus \Z^{5}\ar[r]^{\,\,\,\,\,\,\,\,\,B_2'}\ar[d]_{\pi_2}&\\
   \Z^{5}\ar[r]_{M_2}\ar[ru]_{\!\!\!M}& \Z^{5}\ar[r]_{M_2}\ar[ru]_{\!\!\!M}& \Z^{5}\ar[r]_{M_2}\ar[ru]_{\!\!\!M}& \Z^{5}\ar[r]_{M_2}&.\\
 }$$
The direct limit of the bottom row was computed in Example \ref{ex:fivebyfivemainmatrix}, hence the lemma.
 \end{proof}

Hence the cohomology of the hull $\Omega$ is
$H^2(\Omega)=\Z[\frac16] \oplus \Z^4$. In summary,

\begin{thm}\label{t:cohomologyofOmega}
The cohomology of the continuous hull $\Omega$ is
\begin{eqnarray*}
  &&H^0(\Omega)=\Z\\
\nonumber  &&H^1(\Omega)=0\\
\nonumber  &&H^2(\Omega)=\Z[\tfrac16] \oplus \Z^4.
\end{eqnarray*}

\end{thm}

That the first cohomology group $H^1(\Omega)=0$  is the zero group is very surprising.
The reason is that tiling spaces are very often Cantor fibre
bundles over tori. This usually puts at least a copy of $\Z^2$ into $H^1$. See \cite{SadunTilingSpacesAreCSFB}.
By Equation (\ref{e:ktheory}) we have the following proposition.
\begin{pro}
The K-theory of the continuous hull $\Omega$ is
\begin{eqnarray*}
  &&K^0(\Omega)=\Z[\tfrac16]\oplus\Z^5\\
  &&K^1(\Omega)=0.
\end{eqnarray*}

\end{pro}


\begin{thebibliography}{99}

\bibitem{RiemSurfBeardonBook}
  A. F. Beardon,
  \emph{ A Primer on Riemann Surfaces.}
  Cambridge University Press,
  1984.

\bibitem{Bellissard06}
  Jean Bellissard, Riccardo Benedetti, and Jean-Marc Gambaudo,
  \emph{Spaces of Tilings, Finite Telescopic Approximations
and Gap-Labeling.}
  Commun. Math. Phys. 261,
  (2006) 1-41.



\bibitem{AConnes}
  A. Connes,
  \emph{ Non-commutative Geometry.}\\
 Academic Press,
 San Diego  (1994).

\bibitem{FloydFiniteSubdivisionRules01}
   J. W. Cannon, W. J. Floyd, and W. R. Parry,
   \emph{Finite subdivision rules.}\\
   http://www.math.vt.edu/people/floyd/research/papers/fsr.pdf,
   (2001).


\bibitem{Hatcher02}
  Allen Hatcher,
  \emph{ Algebraic Topology.}
  Cambridge University Press,
  2002.


\bibitem{TopologyJanich84}
   K. J\"anich, S. Levy,
   \emph{Topology.}
   Springer-Verlag New York Inc.,
   (1984).




\bibitem{May99}
  J. P. May,
  \emph{ A concise course in Algebraic Topology.}
  The Univeristy of Chicago Press.
  Chicago and London,
  1999.


\bibitem{Mozes89}
   Shahar Mozes,
   \emph{Tilings, substitution systems and dynamical systems generated by them.}
   Journal D'analyse Mathematique,
   (1989).



\bibitem{FinDimApproxCstaralgebrasBrownOzawa}
   Nathanial Patrick Brown, Narutaka Ozawa,
  \emph{ C*-Algebras and Finite-Dimensional Approximations.}
  American Mathematical Society,
  2008.

\bibitem{Putnam00OrderedKtheory}
 Ian F. Putnam,
 \emph{The ordered $K$-theory of $C^*$-algebras associated with substitution tilings.}
 Commun. Math. Phys. 214,
 (2000) 593-605.

\bibitem{PutnamBible95}
  Jared E. Anderson and Ian F. Putnam.
  \emph{Topological Invariants for Substitution Tilings and their Associated $C^*$-algebras.}
  Department of Mathematics and Statistics, University of Victoria, Victoria  B.C. Canada.
  (1995) 1-45.



\bibitem{PutnamCstarKtheory00}
  Johannes Kellendonk and Ian F. Putnam,
  \emph{Tilings, $C^*$-algebras and $K$-theory.}
  Directions in mathematical quasicrystals, CRM Monogr. Ser., 13, Amer. Math. Soc., Provicence, RI
  (2000) 177-206.




\bibitem{MRSnonFLCpentTiling}
   Maria Ramirez-Solano,
  \emph{ A non FLC regular pentagonal tiling of the plane.}\\
   arXiv:1303.2000,
  2013.

\bibitem{MRSdiscretehull}
   Maria Ramirez-Solano,
  \emph{ Construction of the discrete hull for the combinatorics of a  regular pentagonal tiling of the plane.}
   arXiv:1303.5375,
  2013.

\bibitem{MRScontinuoushull}
   Maria Ramirez-Solano,
  \emph{  Construction of the continuous hull for the combinatorics of a regular pentagonal tiling of the plane.}
   arXiv:1303.5676,
  2013.

\bibitem{MRSinverselimit}
   Maria Ramirez-Solano,
  \emph{Continuous hull of a combinatorial pentagonal tiling as an inverse limit.}
   	arXiv:1304.2652,
  2013.


\bibitem{Rob91}
  E. Arthur Robinson, Jr.,
  \emph{ Symbolic Dynamics and Tilings of $\R^d$.}
  George Washington University.
  Washington DC.
  1991.


\bibitem{Sadun08}
  Lorenzo Sadun,
  \emph{ Topology of Tiling Spaces.}
  University Lecture Series Vol. 46,
  Providence, Rhode Island,
  2008.

\bibitem{SadunTilingSpacesAreCSFB}
   Lorenzo Sadun, R. F. Williams,
  \emph{ Tiling Spaces Are Cantor Set Fiber Bundles.}\\
   http://arxiv.org/pdf/math/0105125.pdf\\
  2001.


\bibitem{Sol98}
  B. Solomyak,
  \emph{ Nonperidicity implies unique composition of self-similar translationally-finite tilings.}
  Disc. Comp. Geom.
  1998.

\bibitem{StephensonBowers97}
  Philip L. Bowers and Kenneth Stephenson,
  \emph{A "regular" pentagonal tiling of the plane.}
  Conformal geometry and dynamics.
  An electronic journal of the American Mathematical Society,
  (1997) 58-86.




\end{thebibliography}
\end{document}